\documentclass[twoside,a4paper]{amsart}
\usepackage{amssymb, amsmath,latexsym, amsthm}
\usepackage{fancyhdr}
\usepackage{geometry,amsmath,amssymb,bbm,latexsym}
\usepackage{hyperref}
\usepackage{marginnote,amssymb,mathrsfs,latexsym,verbatim,pdfsync,color}
\catcode`\<=\active \def<{
\fontencoding{T1}\selectfont\symbol{60}\fontencoding{\encodingdefault}}
\catcode`\>=\active \def>{
\fontencoding{T1}\selectfont\symbol{62}\fontencoding{\encodingdefault}}

\newtheoremstyle{dotless}{}{}{\itshape}{}{\bfseries}{}{ }{}
\theoremstyle{dotless}
\newtheorem{proposition}{Proposition}[section]
\newtheorem{corollary}{Corollary}[section]
\newtheorem{theorem}{Theorem}[section]
\newtheorem{lemma}{Lemma}[section]

\renewenvironment{proof}[1][Proof]{\textbf{#1.} }{\ \rule{0.5em}{0.5em}}

\theoremstyle{definition}

\DeclareMathOperator{\supp}{supp}

\DeclareMathOperator{\capp}{Cap}

\begin{document}

\title[Trace inequalities for the multi-parameter Hardy operators on poly-trees]{Some properties related to trace inequalities for the multi-parameter Hardy operators on poly-trees}

\author{Nicola Arcozzi, Pavel Mozolyako, Karl-Mikael Perfekt}

\thanks{The results of Section \ref{SS:2.4} are supported by the Russian Science Foundation under the grant 17-11-01064}
\thanks{N. Arcozzi is partially supported by the grants INDAM-GNAMPA 2017 "Operatori e disuguaglianze integrali in spazi con simmetrie" and PRIN 2018 "Variet\`{a} reali e complesse: geometria, topologia e analisi armonica"}


\begin{abstract}
In this note we investigate the multi-parameter Potential Theory on the weighted $d$-tree (Cartesian product of several copies of uniform dyadic tree), which is connected to the discrete models of weighted Dirichlet spaces on the polydisc. We establish some basic properties of the respective potentials, capacities and equilibrium measures (in particular in the case of product polynomial weights). We explore multi-parameter Hardy inequality and its trace measures, and discuss some open problems of potential-theoretic and combinatorial nature.
\end{abstract}

\maketitle
\date{\today}
\section{Introduction}\label{S:1}
The Hardy operator on the set of the positive integers is given by $I\varphi(n)=\sum_{l=0}^n\varphi(l)$. The problem of characterizing the positive weights 
$u,v:{\mathbb N}\to{\mathbb R}^+$ such that
\begin{equation}\label{hardy1}
\|I\varphi\|_{\ell^p(v)}\le C\|\varphi\|_{\ell^p(u)}
\end{equation}
for some positive $C$, depending only on $1<p<\infty,u$, and $v$, has long been considered and solved. It was only rather recently \cite{ars2002}
that an analogous problem was considered 
on trees. Let $T$ be a tree with vertex set $V(T)\ni o$, where $o$ is a \it root \rm of $T$, and define, for $\varphi:V(T)\to\mathbb{R}_+$, the function
$I\varphi:V(T)\to\mathbb{R}_+$ as $I\varphi(\alpha)=\sum_{\beta\in[o,\alpha]}\varphi(\beta)$, where $[o,\alpha]$ is the ``geodesic'' joining $\alpha$ to the root.
In fact, $\mathbb{N}$ is a particular tree, but general trees might exhibit the exponential growth, with respect to $n$, of the number of  points having distance $n$ 
to a distinguished vertex. The usual dyadic tree is a typical example. We postpone the precise definition of the Hardy operator on trees to the next section,
where the necessary notation is introduced.

Characterizing the two-weight inequality for the Hardy operator on trees led to a new characterization of the Carleson measures for the Dirichlet space, 
a result originally due to 
Stegenga \cite{stegenga1980}, which applied, in fact, to a wide range of exponents, weights, and underlying spaces. 
In its simpler form, one wants to characterize measure 
$\mu\ge0$ on the unit disc $\{z\in\mathbb{C}:\ |z|<1\}$ in the complex plane, such that
\begin{equation}\label{hardy2}
\iint_{\mathbb{D}}|f(z)|^2d\mu(z)\le C\iint_{\mathbb{D}}|f^\prime(z)|^2dxdy
\end{equation}
for all holomorphic functions satisfying $f(0)=0$. The connection between the holomorphic problem and the discrete one might be summarized as follows.
The function $f$ on $\mathbb{D}$ is somehow identified with the function $I\varphi$, the function $f^\prime$ with $\varphi$ (the ``derivative'' of $I\varphi$),
the unit disc $\mathbb{D}$ with the tree $T$ which indexes its dyadic Whitney decomposition.

In 1985, E. Sawyer \cite{sawyer1985} considered the extension of (\ref{hardy1}) to the bi-linear case: $II\varphi(m,n)=\sum_{i=0}^m\sum_{j=0}^n\varphi(i,j)$, with
$\varphi:\mathbb{N}^2\to\mathbb{R}_+$. He characterized the two weight inequality for the bi-linear Hardy operator $II$, and it should be mentioned that the tri-linear inequality still awaits a 
characterization. 

Recently, we and Giulia Sarfatti \cite{amps2018} considered the problem of characterizing the Carleson measures for the Dirichlet space 
on the bi-disc, which may be thought of $\mathcal{D}(\mathbb{D}^2)\equiv\mathcal{D}(\mathbb{D})\otimes\mathcal{D}(\mathbb{D})$. The first step is reducing the problem to one on the bi-tree: 
the Cartesian product of two copies of the tree, with the corresponding Hardy operator defined by summation on Cartesian products of geodesics, 
as in Sawyer's result. We could not modify the proof of Sawyer, however, to make it work in the bi-tree case. Our proof follows Stegenga's idea of 
proving a \textit{capacitary strong type inequality}, which is the heart of the proof.

In this note, we prove some results in multi-linear potential theory, which might prove useful in extending the results in \cite{amps2018} to (i) polytrees 
(with more than two factors), (ii) with weights. In Section \ref{SS:2.4} we prove that the capacity of a subset $E$ of the polytree $T^d=T\times\dots\times T$
is comparable to that of its projection $\mathcal{S}_b(E)$ onto the \it distinguished boundary \rm $(\partial T)^d$ of $T^d$. The novelty is that we consider
the discrete problem arising from the study  of the Potential Theory associated with \it weighted \rm Dirichlet spaces, 
which have not been so far investigated. 
In Section \ref{oggi}, we give two noncapacitary sufficient conditions for a measure to satisfy the trace inequality for the multilinear Hardy operator on a polytree.\\
Throughout this paper we refer to some basic facts from potential theory, as presented in \cite[Chapter 2]{ah1996}.

\section{Weighted $d$-tree and potential theory}\label{S:2}
\subsection{A $d$-tree}
As in \cite{amps2018} we start by considering the rooted directed (away from the root) uniform infinite binary tree (\textit{a dyadic tree}). The order relation on the vertex set $V(T)$ is given by direction: for $\alpha,\beta\in V(T)$ we say that $\alpha \leq \beta$, if one can get from $\beta$ to $\alpha$ following the directed root. In other words, $\beta$ is one of the endpoint of the edges in the geodesic $[\alpha,o]$ connecting $\alpha$ and the root $o$. We also write $\alpha < \beta$, if $\alpha\leq\beta$, and $\alpha\neq \beta$. The boundary $\partial T$ of  the tree is defined in a standard way; each point $\omega \in \partial T$ is encoded as an infinite directed sequence $[e^0,e^1,\dots]\subset E(T)$ of connected edges that starts at the root $o$ (i.e. $o$ is the endpoint of $e^0$). The order relation makes sense for $\partial T$ as well, given $\omega\in \partial T$  we say that $\omega \leq \alpha$, if and only if  $\alpha$ is an endpoint of one of the edges $e^{k}$ encoding $\omega$, or $\alpha = \omega$. We write $\overline{T} := T\bigcup \partial T$. In what follows we identify the vertex set $V(T)$ and the tree itself, i.e. we assume that $\alpha\in T$ is always a vertex.\\
If $\alpha,\beta \in \overline{T}$, then there there exists a unique point $\gamma \in \overline{T}$ that is \textit{the least common ancestor} of $\alpha$ and $\beta$, we denote it by $\alpha\wedge\beta$. Namely, we have that $\gamma\geq\alpha,\;\gamma\geq\beta$, and if there is another point $\tilde{\gamma}$ satisfying these relations, then $\tilde{\gamma}\geq\gamma$ (basically $\gamma$ is the first intersection points of geodesics connecting $\alpha$ and $\beta$ to the root). In particular, $\alpha\wedge\alpha = \alpha$. The total amount of common ancestors of $\alpha$ and $\beta$ is denoted by $d_T(\alpha\wedge\beta)$ ( $d_T(\alpha\wedge\beta) = dist_T(\alpha\wedge\beta,o)+1$, where $dist_T$ is the usual graph distance on $T$). $d_T$ can be infinite, for instance, $d_T(\omega\wedge\omega) = \infty$ when $\omega\in\partial T$. The \textit{predecessor set} (with respect to the geometry of $\overline{T}$) of a point $\alpha\in V(T)\cup\partial T$ is
\[
\mathcal{P}(\alpha) = \{\beta\in \overline{T}:\; \beta\geq\alpha\}.
\]
In particular, every point is its own predecessor. The \textit{successor set} is
\[
\mathcal{S}(\beta) := \{\alpha\in\overline{T}:\; \beta\in\mathcal{P}(\alpha)\},\quad \beta\in \overline{T}.
\]
Clearly $d_T(\alpha\wedge\beta) = \sharp\mathcal{P}(\alpha\wedge\beta)$.\par
We are now ready to define the $d$-tree. Fix an integer $d$, and consider $T_1,T_2,\dots, T_d$ --- identical copies of the dyadic tree $T$. The vertex set $V(T)^d$ of the  graph $T^d$ is defined as follows
\begin{equation}\notag
V(T^d):= V(T)^d = V(T_1)\times V(T_2)\times\dots\times V(T_d),
\end{equation}
i.e. $\alpha \in V(T^d)$, if $\alpha = (\alpha_1,\dots,\alpha_d)$ with $\alpha_j\in T_j,\; j=1,\dots,d$. Two vertices $\alpha,\beta\in V(T^d)$ are connected by an edge, if and only if there exists a number $1\leq j\leq d$ such that $\alpha_j$ and $\beta_j$ are connected by an edge in $T_j$, and $\alpha_k = \beta_k$ for any $k\neq j$. As before, we usually identify $V(T^d)$ and $T^d$.\\
The order relation on $T^d$ is induced by the order on its coordinate trees, we say that $\alpha \leq \beta$, if $\alpha_j \leq\beta_j$ for every $1\leq j \leq d$. The boundary of the $d$-tree is
\[
\partial T^d  = \bigcup_{D\subset \{1,2,\dots,d\}}\prod_{j\in D}T_j\prod_{k\in \{1,2,\dots,d\}\setminus D}\partial T_k
\]
(the Cartesian products are taken according to the order of indices). The set  $\partial T_1\times\partial T_2\times\dots\times\partial T_d $ is called a \textit{distinguished boundary} of $T^d$ and denoted by $(\partial T)^d$. We let $\overline{T}^d = T^d\bigcup\partial T^d$. As before, we define predecessor and successor sets of a vertex $\alpha = (\alpha_1,,\dots,\alpha_d)$ using the same notation
\[
\mathcal{P}(\alpha) = \mathcal{P}(\alpha_1)\times\dots \times\mathcal{P}(\alpha_d),\; \mathcal{S}(\alpha) = \mathcal{S}(\alpha_1)\times\dots\times\mathcal{S}(\alpha_d).
\]
Sometimes we specify the dimension writing $\mathcal{S}_{T}(\alpha)$ for a point $\alpha$ in the tree $T$, and $\mathcal{S}_{T^d}(\alpha)$ for a point $\alpha$ in the $d$-tree (same goes for the predecessor sets). The part of $\mathcal{S}(\alpha)$ that lies on the distinguished boundary is denoted by $\partial \mathcal{S}(\alpha)$.\\
Similar to one-dimensional setting we denote the number (possibly infinite) of common ancestors of $\alpha$ and $\beta$ by $d_{T^d}(\alpha\wedge\beta)$, where $\alpha\wedge\beta = (\alpha_1\wedge\beta_1,\dots,\alpha_d\wedge\beta_d)$ is a (unique) least common ancestor (in $T^d$) of $\alpha$ and $\beta$. The predecessor and successor sets are defined as above (and denoted in the same way). We have 
\[d_{T^d}(\alpha\wedge\beta) = \prod_{j=1}^d d_T(\alpha_j\wedge\beta_j) = \sharp\mathcal{P}(\alpha\wedge\beta).
\]
We also write $d_T(\alpha_j)$ and $d_{T^d}(\beta)$ instead of $d_T(\alpha_j\wedge\alpha_j)$ and $d_{T^d}(\beta\wedge\beta)$.  \\

\subsection{Potential theory on $d$-tree}\label{SS:2.35}
Before we introduce the basics of potential theory on the $d$-tree we  adapt our space to the conventions used in \cite{ah1996}.

First we define a metric on $\overline{T}_j$: given $\alpha_j,\beta_j\in \overline{T}_j$ we set
\begin{equation}\notag
\delta_j(\alpha_j,\beta_j) := 2^{-d_T(\alpha_j\wedge\beta_j)} - \frac12\left(2^{-d_T(\alpha_j)} + 2^{-d_T(\beta_j)}\right),
\end{equation}
essentially this is a distance associated to the graph distance on $T$ with weights $2^{-dist_T(\alpha_j,o)}$. Then we let
\begin{equation}\label{e:31}
\delta(\alpha,\beta) = \sum_{j=1}^d\delta_j(\alpha_j,\beta_j),\quad \alpha,\beta\in \overline{T}^d.
\end{equation}
Clearly, $\delta$ is a metric on $\overline{T}^d$.\\
We suggest two ways of interpreting a $d$-tree, the first one is less natural in a sense, but it allows us to properly use the machinery in \cite{ah1996}. The dyadic tree is a planar graph, and one can embed it into $\mathbb{R}^2$ in such a way that its boundary $\partial T_j$ is actually a classical ternary Cantor set $E_c$ on the unit interval. As a result we can assume that $\overline{T}_j\subset \mathbb{R}^2$, moreover, embedded with $\delta_j$ it is a locally compact Radon space, and Borel sets in $\overline{T}_j$ are Borel in $\mathbb{R}^2$.
In the same vein, the points of $T^d$ embed into $\mathbb{R}^{2d}$. In particular $(\partial T)^d$ can be identified with $E_c^2$.\par 

Let $\pi$ be a positive Borel measure on $T^d$, that is, collection of positive weights on vertices of $T^d$ -- we always assume $\pi$ has zero mass on $\partial T^d$. Denote by $\mathbb{M}$ the (open) $d$-tree $T^d$ equipped with  measure $\pi$ and a family of Borel (with respect to the distance $\delta$) measurable sets. We define a kernel $G:\mathbb{R}^{2d}\times\mathbb{M}\rightarrow \mathbb{R}_+$ to be $G(\alpha,\beta) := \chi_{\mathcal{S}_{\beta}}(\alpha)$, where $\alpha\in \overline{T}^d\subset\mathbb{R}^{2d}$, $\beta\in T^d$ and $\mathcal{S}_{\beta} := \{\gamma\in\overline{T}^d:\; \gamma\leq\beta\}$ is the $\overline{T}^d$-successor set of $\beta$. It is easy to verify that $G$ is lower semicontinuous on $\overline{T}$ in first variable, and measurable on $\mathbb{M}$ in second variable. This means that we are now squarely in the context of Adams and Hedberg (\cite[Chapter 2.3]{ah1996}), and we can proceed with the Potential Theory. Given a non-negative Borel measure $\mu$ on $\overline{T}^d$ (which, again, is by extension Borel on $\mathbb{R}^{2d}$) and a non-negative $\pi$-measurable function $f$ on $\mathbb{M}$ we let
\begin{subequations}\label{e:32}
\begin{eqnarray}
\label{e:32.1}& (\mathbb{I}f)(\alpha) := \int_{\mathbb{M}}G(\alpha,\beta)f(\beta)\,d\pi(\beta) = \sum_{\gamma\geq\alpha}f(\gamma)\pi(\gamma),\\
\label{e:32.2}& (\mathbb{I}^*\mu)(\beta) := \int_{\bar{T}^2}G(\alpha,\beta)\,d\mu(\alpha) = \int_{\mathcal{S}(\beta)}\,d\mu(\alpha).
\end{eqnarray}
\end{subequations}
Observe that a measure supported on $T^2$ and a non-negative function are pretty much the same objects --- a collection of masses assigned to the points of the $d$-tree. The Potential Theory generated by these two operators leads us to the notions of $\pi$-potential 
\begin{equation}\label{e:33}
\mathbb{V}_{\pi}^{\mu} := (\mathbb{I}\mathbb{I}^*)(\mu)
\end{equation}
and capacity
\begin{equation}\label{e:334}
\capp_{\pi} E := \inf\left\{\int f^2\,d\pi:\; f\geq0,\,(\mathbb{I}f)(\alpha)\geq1,\; \alpha\in E\right\},\quad E\subset \overline{T}^d.
\end{equation}
Given two Borel measures $\mu,\nu\geq0$ on $\overline{T}^d$ we define their mutual energy to be
\begin{equation}\label{e:337}
\mathcal{E}_{\pi}[\mu,\nu] := \int_{\overline{T}^d}\mathbb{V}^{\mu}\,d\nu = \int_{\overline{T}^d}\mathbb{V}^{\nu}\,d\mu = \sum_{\alpha\in T^d}(\mathbb{I}^*\mu)(\alpha)(\mathbb{I}^*\nu)(\alpha)\pi(\alpha) = \langle\mathbb{I^*}\mu,\mathbb{I}^*\nu\rangle_{L^2(T^d,\pi)},
\end{equation}
the last two equalities following from Tonelli's theorem. When $\mu=\nu$ we write $\mathcal{E}_{\pi}[\mu]$ instead, and we call it \textit{the energy of $\mu$}. Given a Borel set $E\subset \overline{T}^d$ there exists a uniquely defined \textit{equilibrium measure} $\mu_E\geq0$ that generates the minimizer in \eqref{e:334}, so that
\[
\capp_{\pi} E = \int_{T^d}(\mathbb{I}^*\mu_E)^2\,d\pi = \mathcal{E}_{\pi}[\mu_E] = \mu_E(E)
\]
(see \cite{ah1996}). If $E$ is a compact set, then one also has $\supp\mu_E\subset E$.\par
Another way to look at the $d$-tree (which is more convenient and tangible) is the dyadic rectangle representation. It is well known that a dyadic tree can be interpreted as a collection of dyadic subintervals of some basic interval (say, $[0,1]$), with a natural order given by inclusion. This approach is not without its own problems though, since $\partial T$ and $[0,1]$ do not have a one-to-one correspondence --- dyadic-rational points can be images of two different elements of $\partial T$. This obstacle however is is not relevant in the context of the potential theory we have developed, since the measures we are working with do not distinguish these points. In other words, if the measure has finite energy for an appropriate choice of weight $\pi$, its total mass on the non-injective set is zero, see Lemma \ref{l:1}. That means that for every point $\alpha\in T^d$ there exists a unique  dyadic rectangle $R_{\alpha} = \prod_{j=1}^d[k_j2^{-n_j},(k_j+1)2^{-n_j}]\subset [0,1]^d$ with $n_j\geq0$ and $0\leq k_j\leq 2^{n_j}-1$, and vice-versa, every such dyadic rectangle corresponds to a point $\alpha\in T^d$. In the same way, the \textit{distinguished boundary} can be roughly viewed as the unit cube $[0,1]^d$ (again, the problematic points are not seen by finite energy measures). The rest of $\partial T^d$ is visualized similarly. \\
This representation makes it clear that a $d$-tree (for $d\geq 2$) is \textit{not a tree}, since, for instance, every point has several geodesics connecting it to the root $(o_1,\dots,o_d)$, and $T^d$ has a lot of cycles. However $T^d$ still have some structural properties inherited from the geometry of $T$, in particular \textit{it does not have any directed cycles}. This allows us to salvage some of the arguments used in one-dimensional case.\\
As usual, we write $A \lesssim B$ if there exists a constant $C$ (that depends only on $d$, $\pi$, and whose value may change from line to line) such that $A \leq CB$, and $A\approx B$, if $A\lesssim B$ and $B\lesssim A$.

\section{Properties of potentials and standard polynomial weights}\label{SS:2.4} 
\subsection{Basic properties of potentials and capacities}
We call $\pi$ a product weight, if $\pi(\beta) = \prod_{j=1}^n\pi_j(\beta_j)$, where $\pi_j$ is a weight on $T_j$.
\begin{lemma}\label{l:1}
Assume $\pi$ is bounded away from zero. Then the following properties hold:
\begin{enumerate}
\item If $\mu\geq0$ is a Borel measure on $(\partial T)^d$ with finite energy, then $\mu (\{\omega_j\}\times\prod_{k\neq j}\partial T_k) = 0$ for any $\omega_j \in \partial T_j,\; 1\leq j\leq d$.\label{l:A1.s.1}
\item Assume $\pi$ is a product weight. Let $E\subset \overline{T}^d$ be a Borel set. Define $E_j\subset \overline{T}_j$ to be its coordinate projections, i.e. $\alpha_j\in E_j$, if there exist points $\alpha_k\in \overline{T}_k,\; k\neq j$ such that $(\alpha_1,\dots,\alpha_j,\dots,\alpha_d) = \alpha \in E$. Then
\begin{equation}\label{e:l1.1}
\capp_{\pi} E \leq \prod_{j=1}^d \capp_{\pi_j}E_j.
\end{equation}
In particular, if $E$ is a product set, $E = E_1\times\dots\times E_d$, then we have equality in \eqref{e:l1.1}.
\label{l:A1.s.0}
\item  Let $E$ be a Borel subset of $\overline{T}^d$ and $\mu\geq0$ be a Borel measure on $\overline{T}^d$ such that $\mathcal{E}_{\pi}[\mu]<\infty$ and $\mathbb{V}_{\pi}^{\nu} \leq 1$ q.a.e. on $E$. Then $\capp_{\pi} E \geq \mu(E)$. \label{l:A1.s.2}
\item  Let $E$ be a Borel subset of $\overline{T}^d$ and $\mu\geq0$ be a Borel measure on $\overline{T}^d$ with finite mass such that $\mu(E) \geq \mathcal{E}_{\pi}[\mu]$. Then $\capp_{\pi} E \geq \mu(E)$. \label{l:A1.s.2.5}

\end{enumerate}
\end{lemma}
\begin{proof}
\textit{Property \ref{l:A1.s.1}.}\;\;
Assume $j=1$ and $\mu(\{\omega_1\}\times\prod_{k=2}^d\partial T_k) =\varepsilon >0$ for some $\omega_1\in \partial T_1$. Then we immediately have $(\mathbb{I}^*\nu)(\alpha_1,o_2,\dots,o_d) \geq \varepsilon$ for any $\alpha_1>\omega_1$, and
\[
\mathcal{E}_{\pi}[\mu] = \sum_{\alpha\in T^d}(\mathbb{I}^*\nu)^2(\alpha)\pi(\alpha) \geq \sum_{\alpha_1>\omega_1}(\mathbb{I}^*\nu)^2(\alpha_1,o_2,\dots,o_d)\pi(\alpha_1,o_2,\dots,o_d) = \infty,
\]
since $\pi \geq 1$ (actually that is the only thing we need from the weight here).\par
\textit{Property \ref{l:A1.s.0}.}\;\; For every $1\leq j\leq d$ let $f_j$ be some admissible for $E_j$ function (so that\\ $\sum_{\beta_j\geq\alpha_j}f_j(\beta_j)\pi_{j}(\beta) \geq 1$ for every $\alpha_j\in E_j$). Define
\[
f(\beta) = \prod_{j=1}^df_j(\beta_j),\; \beta\in T^d.
\]
Since $\pi = \prod_{j=1}^d\pi_j$, we clearly have
\[
\sum_{\beta\geq\alpha}f(\beta)\pi(\beta) = \prod_{j=1}^d\sum_{\beta_j\geq\alpha_j}f(\beta_j)\pi(\beta_j)
\]
for any $\alpha\in \overline{T}^d$. Therefore $f$ is admissible for $\prod_{j=1}^d E_j$. In the same fashion, $\|f\|^2_{L^2(T^d,d\pi)} = \prod_{j=1}^d\|f_j\|_{L^2(T_j,d\pi_j)}^2$, hence $\capp_{\pi}\left(\prod_{j=1}^dE_j\right) \leq \prod_{j=1}^d\capp_{\pi_j}E_j$, and \eqref{e:l1.1} follows immediately. In particular, the product of polar sets is a polar set as well.\\
To get the equality for product sets we turn to the dual definition of capacity: 
\[
\capp_{\pi_j} = \sup\{\mu_j(E_j)^2:\;\supp\mu_j\subset E_j,\; \mathcal{E}_{\pi_j}[\mu_j]\leq 1\}.
\]
Now let $\mu_j$ be some admissible (in the sense above) measure for $E_j$. Define $\mu$ to be the usual extension of $\prod_{j=1}^d\mu_j$ to $\overline{T}^d$. As before, for any $\alpha\in \overline{T}^d$ one has
\begin{equation}\notag
\begin{split}
&\mathbb{V}_{\pi}^{\mu}(\alpha) = \sum_{\beta\geq\alpha}(\mathbb{I}^*\mu)(\beta)\pi(\beta) = \sum_{\beta\geq\alpha}\int_{\mathcal{S}(\beta)}\,d\mu\pi(\beta) = \int_{\overline{T}^d}\sum_{\beta\geq\alpha}\,\pi(\beta)\chi_{\mathcal{S}(\beta)}(\tau)d\mu(\tau) =\\ &\int_{\overline{T}^d}\sum_{\beta\geq\alpha\wedge\tau}\,\pi(\beta)d\mu(\tau) =
 \int_{\overline{T_1}}\dots\int_{\overline{T}_d}\sum_{\beta_1\geq\alpha_1\wedge\tau_1}\pi(\beta_1)\cdot.\,.\,.\,\cdot\sum_{\beta_1\geq\alpha_d\wedge\tau_d}\pi(\beta_d)\,d\mu_1(\tau_1)\dots\,d\mu_d(\tau_2) = \prod_{j=1}^{d}V_{\pi_j}^{\mu_j}(\alpha).
\end{split}
\end{equation}
In particular, we observe that $\mathcal{E}_{\pi}[\mu] = \prod_{j=1}^d\mathcal{E}_{\pi_j}[\mu_j]$, hence $\pi$-energy of $\mu$ is less than $1$. Combined with the fact that $\supp\mu\subset \prod_{j=1}^dE_j$ we obtain $\capp_{\pi}\left(\prod_{j=1}^dE_j\right) \geq \prod_{j=1}^d\capp_{\pi_j}E_j$.\par
\textit{Property \ref{l:A1.s.2}}.\;\; Define, as usual, the restricted measure $\mu\vert_E$ by $\mu\vert_E(F) := \mu(E\bigcap F)$, and let $\mu_E$ be the equilibrium measure of $E$. Clearly, $\mathcal{E}_{\pi}[\mu_E] < \infty$, and $\mathbb{V}_{\pi}^{\mu\vert_E} \leq \mathbb{V}_{\pi}^{\mu}$. We have
\[
\mathcal{E}_{\pi}[\mu\vert_E] = \int_{\overline{T}^d}\mathbb{V}_{\pi}^{\mu\vert_E}\,d\mu\vert_E \leq \int_{\overline{T}^d}\mathbb{V}_{\pi}^{\mu_E}\,d\mu\vert_E = \mathcal{E}_{\pi}[\mu_E,\mu\vert_E],
\]
since $\mathbb{V}_{\pi}^{\mu_E}\geq1$ q.a.e. on $E$. Hence, by positivity of the energy integral,
\[
\mu(E) = \mu\vert_E(E) = \int_{\overline{T}^d}\,d\mu\vert_E \leq \mathcal{E}_{\pi}[\mu_E,\mu\vert_E] \leq \mathcal{E}_{\pi}[\mu_E] = \capp E.
\]
\textit{Property \ref{l:A1.s.2.5}},\;\; Let $\mu_E$ be the equilibrium measure of $E$. Clearly
\[
\mathcal{E}_{\pi}[\mu_E,\mu] = \int_{\overline{T}^d}\mathbb{V}^{\mu_E}_{\pi}\,d\mu \geq \mu(E) \geq \mathcal{E}_{\pi}[\mu].
\]
By positivity of energy integral it follows that $\capp_{\pi}E = \mathcal{E}_{\pi}[\mu_E] \geq \mathcal{E}_{\pi}[\mu_E,\mu]$.\par
\end{proof}
\subsection{Standard polynomial weights and capacity of the boundary}
From now on we are restricting ourselves to a special class of weights --- the so-called standard polynomial weights, where $\pi_j(\beta_j) = 2^{s_jd_{T}(\beta_j)}$ for some $0\leq s_j <1$. This class is connected to the discrete representation of weighted Dirichlet space on the polydisc, i.e. space of analytic functions $f$ on $\mathbb{D}^d$, $f(z)=\sum_{a_1,\dots,a_d\geq0}\hat{f}(a_1,\dots,a_d)z^{a_1}_1\dots z^{a_d}_d$, which satisfy that
\[
\sum_{a_1,\dots,a_d}|\hat{f}(a_1,\dots,a_d)|^2(a_1+1)^{1-s_1}\cdot\dots\cdot(a_d+1)^{1-s_d} < +\infty.
\]
In this case there is a natural way to push down a measure $\mu$ defined on the whole $d$-tree to its distinguished boundary $(\partial T)^d$. \\
To do that we first need to define an analogue of Lebesgue measure on $(\partial T)^d$. We start with a dyadic tree $T$. For any point $\alpha\in T$ we put
\[
M(\partial\mathcal{S}(\alpha)) := 2^{-d_T(\alpha)+1}
\]
to be the 'length' of a 'dyadic interval' on $\partial T$. We see that $M$ can be extended to a Borel measure on $\partial T$ satisfying the property above (also, clearly, it has no mass on singletons). Since $M$ is finite, there exists a unique Borel measure $M_d$ on $(\partial T)^d$ such that
\[
M_d(\mathcal{S}(\alpha)\bigcap(\partial T)^d) = \prod M(\partial\mathcal{S}(\alpha_j))
\]
(observe that $M_d(\{\omega_j\}\times\prod_{j\neq k} T_{k})=0$ for any $1\leq j\leq d$ and $\omega_j\in\partial T_j$). \\
Suppose now $\mu\geq0$ is a Borel measure on $\overline{T}^d$ with finite energy. By the disintegration theorem we can define a measure $\mu_b$ supported on the $(\partial T)^d$ to be
\begin{equation}\label{e:lA.3}
\begin{split}
d\mu_b(\omega_1,\dots,\omega_d) := \sum_{D\subset\{1,\dots,d\}}\sum_{j\in D}\sum_{\beta_j>\omega_j}\frac{d\mu(\tau(D,\beta,\omega))}{\prod_{j\in D}M(\partial\mathcal{S}(\beta_j))}\,\prod_{j\in D}dM(\omega_j),
\end{split}
\end{equation}
where $\tau(D,\beta,\omega)_j = \beta_j$, if $j\in D$, and $\tau(D,\beta,\omega)_j = \omega_j$ otherwise. Roughly speaking, we take the mass $\mu(\beta)$ and distribute it uniformly over $\mathcal{S}(\beta)$, the distinguished boundary part of $\mu$ we leave as it is, and we do a mixed distribution on the rest of $\partial T^d$.

\begin{theorem}\label{t:1}
The potentials of $\mu$ and $\mu_b$ are equivalent,
\begin{equation}\label{e:lA.3s}
\mathbb{V}_{\pi}^{\mu}(\alpha) \approx \mathbb{V}_{\pi}^{\mu_b}(\alpha),\quad \alpha\in \overline{T}^d.
\end{equation}
\end{theorem}

\textit{Proof of Theorem \ref{t:1}}.\\
Given $\alpha,\beta\in \overline{T}^d$ define $d_{\pi}(\alpha\wedge\beta) := \sum_{\gamma\geq\alpha\wedge\beta}\pi(\gamma)$. Since $\mathbb{V}^{\mu_b} = \int_{(\partial T)^d}d_{\pi}(\zeta\wedge\omega)\,d\mu_b$, we want to compare the values of $d_{\pi}(\alpha\wedge\beta)$, and the average of $d_{\pi}$ taken over the boundary projections of $\alpha$ and $\beta$.
\begin{lemma}\label{l:51}
One has
\begin{equation}\label{e:91}
d_{\pi}(\alpha\wedge\beta) \approx \frac{1}{M_d(\partial\mathcal{S}(\alpha))M_d(\partial\mathcal{S}(\beta))}\int_{\partial\mathcal{S}(\alpha)}\int_{\partial\mathcal{S}(\beta)}d_{\pi}(\xi\wedge\omega)dM_d(\xi)\,dM_d(\omega)
\end{equation} 
($d_{\pi}$ is almost a martingale with respect to the measure $M$).
\end{lemma}
\begin{proof}
Due to multiplicativity it is enough to prove that, say,
\begin{equation}\notag
d_{\pi_1}(\alpha_1\wedge\beta_1) \approx \frac{1}{M(\partial\mathcal{S}(\alpha_1))M(\partial\mathcal{S}(\beta_1))}\int_{\partial\mathcal{S}(\alpha_1)}\int_{\partial\mathcal{S}(\beta_1)}d_{\pi_1}(\xi_1\wedge\omega1)dM(\xi_1)\,dM(\omega_1).
\end{equation}
If $\xi_1\leq\alpha_1$ and $\omega_1\leq\beta_1$, then $d_{\pi_1}(\xi_1\wedge\omega_1) \geq d_{\pi_1}(\alpha_1\wedge\beta_1)$, hence
\[
d_{\pi_1}(\alpha_1\wedge\beta_1) \leq \frac{1}{M(\partial\mathcal{S}(\alpha_1))M(\partial\mathcal{S}(\beta_1))}\int_{\partial\mathcal{S}(\alpha_1)}\int_{\partial\mathcal{S}(\beta_1)}d_{\pi_1}(\xi_1\wedge\omega_1)dM(\xi_1)\,dM(\omega_1).
\]
To get the reverse inequality we first show that for any $\beta_1\in T_1$ and $\tau_1\in \overline{T}_1$ we have
\begin{equation}\label{e:92}
\frac{1}{M(\partial\mathcal{S}(\beta_1))}\int_{\partial\mathcal{S}(\beta_1)}d_{\pi_1}(\tau_1\wedge\omega_1)\,dM(\omega_1) \lesssim d_{\pi_1}(\tau_1\wedge\beta_1).
\end{equation}
If $\tau_1\geq\beta_1$ or these two points are not comparable, then, clearly, $d_{\pi_1}(\tau_1\wedge\beta_1) = d_{\pi_1}(\tau_1\wedge\omega_1)$ for $\omega_1\leq\beta_1$, and \eqref{e:92} is trivial. Hence from  now on we assume that $\tau_1 < \beta_1$. First we note that since $\pi$ is a standard polynomial weight, one has $d_{\pi_1}(\gamma) \approx d_T(\gamma)$, if $s_1d_T(\gamma)\leq1$, and $d_{\pi_1}(\gamma) \approx \frac{1}{s_1}2^{s_1d_T(\gamma)}$, if $s_1d_T(\gamma)\geq1$, for any $\gamma\in \overline{T}_1$. Let $n:= d_T(\beta_1)$ and $N:= d_T(\tau_1)$. For every $n\leq k \leq N$ there exists exactly one point $\gamma_k\in T_1$ such that $\tau_1 \leq\gamma_k \leq\beta_1$, and $d_T(\gamma_k) = k$ (in particular $\gamma_1 = \beta_1,\; \gamma_N = \tau_1$). Define 
\[
S_k = \partial\mathcal{S}(\gamma_k)\setminus\partial\mathcal{S}(\gamma_{k+1}),\quad n\leq k \leq N-1,
\]
and
\[
S_N = \partial\mathcal{S}(\tau_1).
\]
If $\omega_1\in S_k$, then, clearly, $d_{\pi_1}(\tau_1\wedge\omega_1) \approx \frac{1}{s_1} 2^{s_1k}$ for $k\geq\frac{1}{s_1}$, and $d_{\pi_1}(\tau_1\wedge\omega_1) \approx k$ otherwise. Moreover, these sets are disjoint and form a covering of $\partial\mathcal{S}(\beta_1)$. Also $M(S_k) = 2^{-k}- 2^{-k-1}, n\leq k \leq N-1$ and $M(S_N) = 2^{-N}$. We have
\begin{equation}\notag
\begin{split}
&\frac{1}{M(\partial\mathcal{S}(\beta_1))}\int_{\partial\mathcal{S}(\beta_1)}d_{\pi_1}(\tau_1\wedge\omega_1)\,dM(\omega_x) =\\
& 2^{d_T(\beta_1)}\sum_{k=n}^N\int_{S_k}d_{\pi_1}(\tau_1\wedge\omega_1)\,dM(\omega_1) \approx
2^n\sum_{k=n}^{\left[\frac{1}{s_1}\right]}k\cdot M(S_k)+ \frac{1}{s_1}2^n\sum_{k=\max\left(n,\left[\frac{1}{s_1}\right]\right)}^N2^{s_1k}\cdot M(S_k) \leq\\
& 2^n\sum_{k=n}^{\left[\frac{1}{s_1}\right]}k2^{-k} + \frac{1}{s_1}2^n\sum_{k=\max\left(n,\left[\frac{1}{s_1}\right]\right)}^N2^{-k(1-s_1)}\leq \frac{10}{1-s_1}d_{\pi_1}(\tau_1\wedge\beta_1),
\end{split}
\end{equation}
and we arrive at \eqref{e:92}. It follows immediately that
\begin{equation}\notag
\begin{split}
&\frac{1}{M(\partial\mathcal{S}(\alpha_1))M(\partial\mathcal{S}(\beta_1))}\int_{\partial\mathcal{S}(\alpha_1)}\int_{\partial\mathcal{S}(\beta_1)}d_{\pi_1}(\xi_1\wedge\omega_1)\,dM(\xi_1)\,dM(\omega_1) \leq\\
&\frac{10}{1-s_1}\frac{1}{M(\partial\mathcal{S}(\alpha_1))}\int_{\partial\mathcal{S}(\alpha_1)}d_{\pi_1}(\xi_1\wedge\beta_1)M(\xi_1) \leq \frac{100}{(1-s_1)^2} d_{\pi_1}(\alpha_1\wedge\beta_1).
\end{split}
\end{equation}
\end{proof}

We proceed with the proof of Theorem \ref{t:1}.
Fix any point $\alpha\in \overline{T}^d$. We have
\begin{equation}\notag
\begin{split}
&\mathbb{V}^{\mu_b}(\alpha) = \int_{(\partial T)^d}\prod_{j=1}^dd_{\pi_j}(\alpha_j\wedge\omega_j)\,d\mu_b(\omega_1,\dots,\omega_d).
\end{split}
\end{equation}
Consider the first term (the one corresponding to values of $\mu$ on $T^d$) in the expression for $d\mu_b$.
By Tonelli's theorem and Lemma \ref{l:51} one has
\begin{equation}\notag
\begin{split}
&\int_{(\partial T)^d}\prod_{j=1}^dd_{\pi_j}(\alpha_j\wedge\omega_j)\,\sum_{\beta\geq\omega}\frac{\mu(\beta)}{\prod_{j=1}^dM(\mathcal{S}(\beta_j))}\,dM(\omega_1)\dots\,dM(\omega_d) = \\
&\sum_{\beta\in T^d}\mu(\beta)\cdot\prod_{j=1}^d\left(\frac{1}{M(\mathcal{S}(\beta_j))}\int_{\mathcal{S}(\beta_j)}d_{\pi_j}(\alpha_j\wedge\omega_j)\,\,dM(\omega_j)\right) \approx\\
&\sum_{\beta\in T^d}\mu(\beta)d_{\pi}(\alpha\wedge\beta).
\end{split}
\end{equation}
Similarly, if we take one of the mixed terms in \eqref{e:lA.3}, say with $D = \{2,3,\dots,d\}$, we obtain
\begin{equation}\notag
\begin{split}
&\int_{(\partial T)^d}\prod_{j=1}^dd_{\pi_j}(\alpha_j\wedge\omega_j)\,\sum_{j=2}^d\sum_{\beta_j\geq\omega_j}\frac{d\mu(\omega_1,\beta_2,\dots,\beta_d)}{\prod_{j=2}^dM(\mathcal{S}(\beta_j))}\,dM(\omega_2)\dots\,dM(\omega_d) =\\
&\sum_{j=2}^d\sum_{\beta_j\in T_j}\int_{\partial T_1}d_{\pi_1}(\alpha_1\wedge\omega_1)d\mu(\omega_1,\beta_2,\dots,\beta_d) \cdot\prod_{j=2}^d\left(\frac{1}{M(\mathcal{S}(\beta_j))}\int_{\mathcal{S}(\beta_j)}d_{\pi_j}(\alpha_j\wedge\omega_j)\,dM(\omega_j)\right)\approx\\
&\sum_{j=2}^d\sum_{\beta_j\in T_j}\int_{\partial T_1}d_{\pi_1}(\alpha_1\wedge\omega_1)\prod_{j=2}^dd_{\pi_j}(\alpha_j\wedge\beta_j)d\mu(\omega_1,\beta_2,\dots,\beta_d) = \int_{\partial T_1\times T_2\times\dots\times T_d}d_{\pi}(\alpha\wedge\tau)\,d\mu(\tau).
\end{split}
\end{equation}
The rest of the terms are done in the same way.

We arrive at
\begin{equation}\notag
\begin{split}
&V^{\mu_b}_{\pi}(\alpha) =\\
& \int_{(\partial T)^d}\sum_{D\subset\{1,\dots,d\}}\sum_{j\in D}\sum_{\beta_j>\omega_j}\prod_{j\in D}d_{\pi_j}(\alpha_j\wedge\beta_j)\prod_{j\in \{1,\dots,d\}\setminus D}d_{\pi_j}(\alpha_j\wedge\omega_j)\frac{d\mu(\tau(D,\beta,\omega))}{\prod_{j\in D}M(\partial\mathcal{S}(\beta_j))}\,\prod_{j\in D}dM(\omega_j) \approx\\
& \sum_{D\subset\{1,\dots,d\}}\int_{\prod_{j\in D} T_j\times\prod_{j\in \{1,\dots,d\}\setminus D}}d_{\pi}(\alpha\wedge\tau)\,d\mu(\tau) = \int_{\overline{T}^d}d_{\pi}(\alpha\wedge\tau)\,d\mu(\tau) = \mathbb{V}_{\pi}^{\mu}(\alpha),
\end{split}
\end{equation}
here the Cartesian product is taken according to the order of indices.
$\blacksquare$

\begin{corollary}\label{c:1}
Given a compact set $E\subset \overline{T}^d$ define its boundary projection $\mathcal{S}_b(E)\subset(\partial T)^d$ to be
\[
\mathcal{S}_b(E) = \bigcup_{\beta\in E}\partial\mathcal{S}(\beta).
\]
Then there exists a constant $C>1$ depending only on $d$ and $\pi$ such that
\begin{equation}\label{e:cA.3}
\capp_{\pi} \mathcal{S}_b(E) \leq\capp_{\pi} E \leq C\capp_{\pi} \mathcal{S}_b(E).
\end{equation}
\end{corollary}
\begin{proof}
The left inequality is trivial, since any function admissible for $E$ is also admissible for $\mathcal{S}_b(E)$. Now let $\mu$ and $\nu$ be the equilibrium measures for $E$ and $\mathcal{S}_b(E)$ respectively. By definition of $\mu_b$
\begin{equation}\notag
\begin{split}
&|\mu_b| := \int_{(\partial T)^d}\,d\mu_b = 
\int_{(\partial T)^d}\sum_{D\subset\{1,\dots,d\}}\sum_{j\in D}\sum_{\beta_j>\omega_j}\frac{d\mu(\tau(D,\beta,\omega))}{\prod_{j\in D}M(\partial\mathcal{S}(\beta_j))}\,\prod_{j\in D}dM(\omega_j)=\\
&\sum_{D\subset\{1,\dots,d\}}\sum_{j\in D}\sum_{\beta_j\in T_j}\int_{\prod_{j\in D} \mathcal{S}(\beta_j)}\int_{\prod_{k\in \{1,\dots,d\}\setminus D}\partial T_k}\frac{d\mu(\tau(D,\beta,\omega))}{\prod_{j\in D}M(\partial\mathcal{S}(\beta_j))}\,\prod_{j\in D}dM(\omega_j) =\\
&\sum_{D\subset\{1,\dots,d\}}\int_{\prod_{j\in D} T_j\times\prod_{j\in \{1,\dots,d\}\setminus D}}\,d\mu(\tau) = \int_{\overline{T}^d}d_{\pi}\,d\mu(\tau) =: |\mu|.
\end{split}
\end{equation}
By Theorem \ref{t:1} and equilibrium property
\[
|\mu_b| = |\mu| = \int_{\overline{T}^d}\mathbb{V}_{\pi}^{\mu}\,d\mu \approx \int_{\overline{T}^d}\mathbb{V}_{\pi}^{\mu_b}\,d\mu \approx \int_{\overline{T}^d}\mathbb{V}_{\pi}^{\mu_b}\,d\mu_b.
\]
On the other hand, for every $C\in\mathbb{R}$ we have
\begin{equation}\notag
0 \leq \int_{\overline{T}^d}\mathbb{V}_{\pi}^{\mu_b}\,d\mu_b -2C\int_{\overline{T}^d}\mathbb{V}_{\pi}^{\nu}\,d\mu_b + C^2\int_{\overline{T}^d}\mathbb{V}_{\pi}^{\nu}\,d\nu \leq \int_{\overline{T}^d}\mathbb{V}_{\pi}^{\mu_b}\,d\mu_b -2C|\mu_b| + C^2|\nu|,
\end{equation}
since $\nu$ is equilibrium for $\mathcal{S}_b(E)$ and $\mathbb{V}^{\nu} \geq 1$ q.a.e. 
on $\mathcal{S}_b(E)\supset\supp\mu_b$. Hence, if we take $C$ to be large enough, we obtain
\begin{equation}\notag
0 \leq \int_{\overline{T}^d}\mathbb{V}_{\pi}^{\mu_b}\,d\mu_b -C|\mu_b| + C\left(C|\nu|-|\mu_b|\right)\leq C\left(C|\nu|-|\mu_b|\right).
\end{equation}
Therefore
\[
C\capp \mathcal{S}_b(E) = C|\nu| \geq |\mu_b| = |\mu| = \capp E,
\]
and we get the second half of \eqref{e:cA.3}.
\end{proof}

Note that the condition $s_j<1$ imposed on the standard polynomial weights is essential. Indeed, in the proof of Lemma \ref{l:51} one can see, that if $s_j\geq 1$ for some $j$, then the capacity of $\partial T_j$ (and hence of $(\partial T)^d$) becomes zero. In this case we basically leave the domain of weighted graph Dirichlet spaces and move to Hardy spaces, for which the capacity is a much less convenient instrument. Also, since $\pi$ is uniform, the equilibrium measure of the distinguished boundary $(\partial T)^d$ is actually $C M_d$ with $C = C(d,\pi)$.\\

\section{Hardy inequality on $d$-tree and properties of trace measures}\label{oggi}
Assume $\mu \geq 0$ is a Borel measure on the $\overline{T}^d$, and $f\geq0$ is a function on $T^d$. The multilinear weighted Hardy inequality is
\begin{equation}\label{e:53}
\int_{\overline{T}^d}(\mathbb{I}f)^2\,d\mu \leq C\sum_{\alpha\in T^d}f^2(\alpha)\pi(\alpha) = \|f\|^2_{L^2(T^d,\,d\pi)},
\end{equation}
for some constant $C>0$. A measure $\mu$ is called a \textit{trace measure for Hardy inequality}, if \eqref{e:53} holds for any $f\geq0$ with constant $C = C_{\mu}$ depending only on $\mu$ (and of course on the weight $\pi$ and dimension $d$). There is a vast amount of literature on various types of trace inequalities of the form above (see e.g. \cite{arsw2014}, \cite{ks1988}, \cite{ks1986} and references therein). Here we mostly aim to concentrate on this particular discrete version and investigate the relationship between different necessary and sufficient conditions.\\
Inequality \eqref{e:53} means that the operator $\mathbb{I}$ is bounded when acting from $L^2(T^d,\,d\pi)$ to $L^2(\overline{T}^d,\,d\mu)$. Equivalently, the adjoint operator, which we denote by $\mathbb{I}^*_{\mu}$, is bounded;
\begin{equation}\label{e:54}
\|\mathbb{I}^*_{\mu}g\|^2_{L^2(T^d,\,d\pi)} \leq C\|g\|^2_{L^2(\overline{T}^d,\,d\mu)},
\end{equation}
for any $\mu$-measurable $g\geq0$ on $\overline{T}^d$. Since
\begin{equation}\notag
\begin{split}
&\langle\varphi, \mathbb{I}f\rangle_{L^2(\overline{T}^d,\,d\mu)} = \int_{\overline{T}^d}\varphi(\alpha)\sum_{\beta\geq\alpha}f(\beta)\pi(\beta)\,d\mu(\alpha)= \int_{\overline{T}^d}\varphi(\alpha)\sum_{\beta\in T^d}\chi_{\mathcal{S}(\beta)}(\alpha)f(\beta)\pi(\beta)\,d\mu(\alpha) =\\
&\sum_{\beta\in T^d}\int_{\mathcal{S}(\beta)}\varphi(\alpha)\,d\mu(\alpha)f(\beta)\pi(\beta),
\end{split}
\end{equation}
we clearly have 
\[
\mathbb{I}^*_{\mu}g(\alpha) = \int_{\mathcal{S}(\alpha)}g\,d\mu.
\]
Another reason to consider this inequality is to study the connection between Hardy inequality on $d$-tree and Carleson embedding for weighted Dirichlet-type spaces on the polydisc, which has been well established in \cite{ars2008}, \cite{arsw2014} for $d=1$, and, recently, in \cite{amps2018} for $d=2$ and $\pi\equiv 1$ (unweighted case).\\
We start with the dual inequality \eqref{e:54}. Let $\mu\geq0$ be a Borel measure on $d$-tree with finite energy, and assume for simplicity that $\supp\mu\subset(\partial T)^d$ (one can pass to general case by careful application of Theorem \ref{t:1} above). A set $E\subset (\partial T)^d$ is called rectangular, if $E$ is a union of finite collection of 'dyadic rectangles' on $(\partial T)^d$, in other words there exists a collection of points $\{\alpha^j\}_{j=1}^N$ such that
\[
E = \bigcup_{j=1}^N\partial\mathcal{S}(\alpha^j).
\]
Now fix such a set $E$ and let $g := \chi_{E}$, plugging $g$ into \eqref{e:54} we obtain
\begin{equation}\notag
\int_{\overline{T}^d}\chi_{E}\,d\mu = \mu(E) \gtrsim C_{\mu}\|\mathbb{I}^*_{\mu}\chi_E\|^2_{L^2(T^d,\,d\pi)} = C_{\mu}\sum_{\alpha\in T^d}\left(\int_{\mathcal{S}_{\alpha}}\chi_E\,d\mu\right)^2\pi(\alpha) = C_{\mu}\sum_{\alpha\in T^d} \left(\mu\left(\partial\mathcal{S}(\alpha)\bigcap E\right)\right)^2\pi(\alpha).
\end{equation}
Using the dyadic rectangle interpretation from the end of Section \ref{SS:2.35} we can rewrite this inequality as
\begin{equation}\label{e:55}
\mu(E) \gtrsim C_{\mu}\sum_{Q}\mu\left(Q\bigcap E\right)^2\pi(Q),\quad \textup{for any rectangular}\; E,
\end{equation}
where $Q = Q_{\alpha} = \partial \mathcal{S}(\alpha)$ is the uniquely defined 'dyadic rectangle' representing a point $\alpha\in T^d$, and $\pi(Q_{\alpha}):= \pi(\alpha)$. Clearly, the expression on the right-hand side of  \eqref{e:55} is just $\pi$-energy of $\mu$ restricted on the set $E$, we call this inequality \textit{global charge-energy} condition. Moreover, if we only consider those rectangles $Q$ that are inside $E$, we obtain
\begin{equation}\label{e:56}
\mu(E) \gtrsim C_{\mu}\sum_{Q\subset E}\mu(Q)^2\pi(Q),\quad \textup{for any rectangular}\; E,
\end{equation}
this one is called \textit{local charge-energy} condition (the reasoning being that the right hand side can be viewed as a 'local' $\pi$-energy of $\mu$ on the set $E$).\\
One of the questions we are interested in is whether one of these necessary conditions is also sufficient for \eqref{e:54}. We start with the global charge-energy condition. By Property \ref{l:A1.s.2.5} from Lemma \ref{l:1} one has
\[
\capp_{\pi}E \gtrsim C_{\mu}\mu(E),\quad \textup{for any rectangular}\; E,
\]
that is, $\mu$ is a \textit{subcapacitary} measure. In \cite{amps2018} it was shown (for $d=2$ and unweighted case) that subcapacitary property indeed implies the trace condition \eqref{e:53}. Note that the subcapacitary condition should hold for any rectangular set $E$; if $\mu(E) \leq \capp_{\pi}E$ only for some particular set $E$, it does not necessarily imply \eqref{e:55} for that set.\\
Consider now the weaker local charge-energy condition \eqref{e:56}. In \cite{ahmv2018} it was shown that it still is equivalent to the trace inequality, for $d=2$ and $\pi\equiv1$. For $d=1$ and general $\pi$, a proof can be found in \cite{arsw2014}; see also \cite{ahmv2018} for Bellman function approach. Here we want to present a slightly different approach, based on the maximal function inequality.
\begin{theorem}\label{p:1}
Assume that Borel measure $\mu\geq0,\; \supp\mu\subset (\partial T)^d,$ with finite energy satisfies \eqref{e:56}. Then the trace inequality \eqref{e:54} follows, if the maximal function inequality
\begin{equation}\label{e:57}
\int_{(\partial T)^d}g^2\,d\mu \gtrsim C_{\mu}\int_{(\partial T)^d}(\mathcal{M}_{\mu}g)^2\,d\mu
\end{equation}
holds for any $g\in L^2(\overline{T}^d,\,d\mu)$, where
\[
(\mathcal{M}_{\mu}g)(\beta) = \sup_{\beta \leq\alpha\in T^d}\frac{\int_{Q_{\alpha}}g\,d\mu}{\mu(Q_{\alpha})},\quad \beta\in \overline{T}^d.
\]
\end{theorem}
\begin{proof}
Fix a function $g\in L^2(\overline{T}^d,\,d\mu)$, some $k\in\mathbb{Z}$ and consider the set $E_{k}:= \{\omega\in (\partial T)^d:\; (\mathcal{M}_{\mu}g)(\omega)>2^k\}$. Clearly there exists a sequence $\{\alpha^j_{k}\}_{j=1}^{\infty}$ such that $E_{k} = \bigcup_{j=1}^{\infty}\partial\mathcal{S}(\alpha^j_{k})$. Approximating $E_{k}$ by rectangular sets $E^{n}_{k} = \bigcup_{j=1}^{n}\partial\mathcal{S}(\alpha^j_{k})$ we see that local charge-energy condition implies
\[
\mu (E_k) \gtrsim C_{\mu}\sum_{Q\subset E_k}\mu(Q)^2\pi(Q),\quad k\in\mathbb{Z}.
\]
By distribution function argument and maximal function inequality \eqref{e:57} we have
\begin{equation}\notag
\begin{split}
&C_{\mu}^{-1}\int_{(\partial T)^d}g^2\,d\mu \gtrsim \int_{(\partial T)^d}(\mathcal{M}_{\mu}g)^2\,d\mu \approx \sum_{k\in\mathbb{Z}}2^{2k}\mu(E_k) \gtrsim \sum_{k\in\mathbb{Z}}2^{2k}\sum_{Q\subset E_k}\mu(Q)^2\pi(Q) \gtrsim\\
&\sum_{\beta\in T^2}(\mathcal{M}_{\mu}g)^2(\beta)\mu^2(Q_{\beta})\pi(Q_{\beta}),
\end{split}
\end{equation}
since $\bigcup_{k\in\mathbb{Z}} E_k = (\partial T)^d$.
On the other hand,
\begin{equation}\notag
\begin{split}
&\sum_{\beta\in T^2}(\mathcal{M}_{\mu}g)^2(\beta)\mu^2(Q_{\beta})\pi(Q_{\beta}) = \sum_{\beta\in T^2}\left(\sup_{\beta \leq\alpha}\frac{\int_{Q_{\alpha}}g\,d\mu}{\mu(Q_{\alpha})}\right)^2\mu^2(Q_{\beta})\pi(Q_{\beta}) \geq \sum_{\beta\in T^2}\left(\frac{\int_{Q_{\beta}}g\,d\mu}{\mu(Q_{\beta})}\right)^2\mu^2(Q_{\beta})\pi(Q_{\beta})\geq\\
& \sum_{\beta\in T^2}\left(\int_{Q_{\beta}}g\,d\mu\right)^2\pi(Q_{\beta}) = \|\mathbb{I}^*_{\mu}g\|^2_{L^2(T^d,\,d\pi)}.
\end{split}
\end{equation}
We are done.
\end{proof}

For $d=1$ this proposition solves the problem, since the maximal function operator is obviously bounded. In higher dimensions \eqref{e:57} fails for some measures; due to presence of cycles in $T^d$, several rectangles can have non-trivial intersection. However the counterexamples to \eqref{e:57} that we are aware of are of rather non-subcapacitary nature, that is, all of them also fail to satisfy \eqref{e:56}. Therefore one can ask whether the local charge-energy inequality can be transformed into some sufficient conditions for the maximal function inequality. This might not be straightforward, since \eqref{e:56} and \eqref{e:57} scale differently. 

Another question is connected to the nature of the rectangular sets on which we test the trace inequality. In the one-dimensional case it is sufficient that \eqref{e:56} holds for all single rectangles (dyadic intervals). One would expect that a single box test, \eqref{e:56} for rectangles, is no longer sufficient when $d\geq 2$. One might compare with the description of Carleson measures for the Hardy space on the bidisc \cite{chang1979}, but note we have been discussing the dual inequality of the Hardy inequality \eqref{e:53}. The single box test for \eqref{e:53} is just a subcapacitary condition, and it fails to be sufficient already for $d=1$, since, generally, capacity is not additive. In particular, the single box test for \eqref{e:53} on an unweighted dyadic tree asks that $\mu(Q) \lesssim \frac{1}{\log\frac{1}{|Q|}}$. However, if we ask a little bit more from this single box test, we can obtain sufficient conditions for $\mu$ to satisfy the trace inequalities \eqref{e:53} and \eqref{e:54}.
\begin{proposition}\label{p:2}
Let $\pi$ be a standard polynomial weight, and $\varphi:[0,2]^d\rightarrow\mathbb{R}_+$ be a function, increasing in each variable, such that
\begin{equation}\label{e:p2.1}
\begin{split}
&\int_{[0,2]^d}\frac{\varphi(t_1,\dots,t_d)}{t_1^{1+s_1}\cdot\dots\cdot t_d^{1+s_d}}\,dt_1\dots dt_d < +\infty.
\end{split}
\end{equation}
Then, if $\mu\geq0$ is a Borel measure on $(\partial T)^d$ satisfying
\begin{equation}\label{e:p2.3}
\mu(Q_{\alpha}) \leq \varphi(M(Q_{\alpha_1}),\dots,M(Q_{\alpha_d})),
\end{equation}
for any $\alpha = (\alpha_1,\dots,\alpha_d)\in T^d$, then $\mu$ is a trace measure for Hardy inequality \eqref{e:53}.
\end{proposition}
\begin{proof}
First we show that \eqref{e:p2.3} implies that $\mathbb{V}_{\pi}^{\mu}(\omega) \lesssim 1$ for any $\omega\in (\partial T)^d$. The weight $\pi$ is uniform, in particular, it (along with the estimate \eqref{e:p2.3}) depends only on generation numbers of $\beta$, and not on the location of $\beta$ within $T^d$. Therefore it is enough show the boundedness of potential only for one point, say $\omega = (0,\dots,0)$ (i.e. the point that corresponds to the leftmost geodesic taken on each tree $T_j$,\; $j=1,\dots,d$). For any multiindex $n = (n_1,\dots,n_d)\subset \mathbb{N}^d$ there exists a unique point $\beta(n)\geq\omega$ with those exact generation numbers (i.e. $d(\beta_j(n)) = n_j$ and $M(Q_{\beta_j(n)}) = 2^{-n_j}$), hence
\[
\mathbb{V}^{\mu}_{\pi}(\omega) = \sum_{\beta\geq\omega}\mu(Q_{\beta})\pi(Q_{\beta}) \leq \sum_{n\in\mathbb{N}^d}\varphi(2^{-n_1},\dots,2^{-n_d})2^{s_1n_1+\dots+s_dn_d} \lesssim \int_{[0,1]^d}\frac{\varphi(t_1,\dots,t_d)}{t_1^{1+s_1}\cdot\dots\cdot t_d^{1+s_d}}\,dt_1\dots dt_d.
\]
It follows immediately that for any $g\in L^2(\overline{T}^d,\,d\mu)$ one has
\begin{equation}\notag
\begin{split}
&\|\mathbb{I}^*_{\mu}\|^2_{L^2(T^d,\,d\pi)} = \sum_{\alpha\in T^d}\left(\int_{Q_{\alpha}}g\,d\mu\right)^2\pi(Q_{\alpha}) \leq \sum_{\alpha\in T^d}\int_{Q_{\alpha}}g^2\,d\mu\cdot \mu(Q_{\alpha})\pi(Q_{\alpha})\leq\\
&\int_{(\partial T)^d}\sum_{\alpha\geq \omega}\mu(Q_{\alpha})\pi(Q_{\alpha})g^2(\omega)\,d\mu(\omega) \lesssim C_{\mu}\int_{(\partial T)^d}g^2\,d\mu.
\end{split}
\end{equation}
\end{proof}

\end{document}